\documentclass[a4paper,11pt]{amsart}
\usepackage{amsmath,amssymb,amsfonts,enumerate,amsthm,amscd}

\theoremstyle{plain}
\newtheorem{theorem}{Theorem}[section]
\newtheorem{lemma}[theorem]{Lemma}
\newtheorem{proposition}[theorem]{Proposition}
\newtheorem{corollary}[theorem]{Corollary}
\theoremstyle{definition}
\newtheorem{remark}[theorem]{Remark}
\newtheorem{example}[theorem]{Example}

\begin{document}
\setcounter{section}{0}

\title{Commutative rings with one-absorbing factorization}

\author[A. El Khalfi]{Abdelhaq El Khalfi}
\address{El Khalfi Abdelhaq\\Modelling and Mathematical Structures Laboratory \\ Department of Mathematics, Faculty of Science and Technology of Fez, Box 2202,
University S.M. Ben Abdellah Fez, Morocco.
$$ E-mail\ address:\ abdelhaq.elkhalfi@usmba.ac.ma$$}

\author[M. Issoual]{Mohammed Issoual}
\address{Mohammed Issoual\\Modelling and Mathematical Structures Laboratory \\ Department of Mathematics, Faculty of Science and Technology of Fez, Box 2202,
University S.M. Ben Abdellah Fez, Morocco.
$$ E-mail\ address:\ issoual2@yahoo.fr$$}

\author[N. Mahdou]{Najib Mahdou}
\address{Najib Mahdou\\Modelling and Mathematical Structures Laboratory \\ Department of Mathematics, Faculty of Science and Technology of Fez, Box 2202,
University S.M. Ben Abdellah Fez, Morocco.
$$E-mail\ address:\ mahdou@hotmail.com$$}

\author[A. Reinhart]{Andreas Reinhart}
\address{Andreas Reinhart\\Institut f\"ur Mathematik und wissenschaftliches Rechnen, Karl-Franzens-Universit\"at Graz, NAWI Graz, Heinrichstra{\ss}e 36, 8010 Graz, Austria.
$$E-mail\ address:\ andreas.reinhart@uni\textnormal{-}graz.at$$}


\begin{abstract}
Let $R$ be a commutative ring with nonzero identity. A. Yassine et al. defined in the paper (Yassine, Nikmehr and Nikandish, 2020),  the concept of $1$-absorbing prime ideals as follows: a proper ideal $I$ of $R$ is said to be a $1$-absorbing prime ideal if whenever $xyz\in I$ for some nonunit elements $x,y,z\in R$, then either $xy\in I$ or $z\in\ I$. We use the concept of $1$-absorbing prime ideals to study those commutative rings in which every proper ideal is a product of $1$-absorbing prime ideals (we call them $OAF$-rings). Any $OAF$-ring has dimension at most one and local $OAF$-domains $(D,M)$ are atomic such that $M^2$ is universal.
\end{abstract}

\maketitle

\section{Introduction}

Throughout this paper, all rings are commutative with nonzero identity and all modules are unital. Let $\mathbb{N}$ denote the set of positive integers. For $m\in\mathbb{N}$, let $[1,m]=\{n\in\mathbb{N}\mid 1\leq n\leq m\}$. Let $R$ be a ring. An ideal $I$ of $R$ is said to be {\it proper} if $I\neq R$. The {\it radical} of $I$ is denoted by $\sqrt{I}=\{x\in R\mid x^n\in I$ for some $n\in\mathbb{N}\}$. We denote by ${\rm Min}(I)$ the {\it set of minimal prime ideals} over the ideal $I$. The concept of prime ideals plays an important role in ideal theory and there are many ways to generalize it.

In \cite{B} Badawi introduced and studied the concept of $2$-absorbing ideals which is a generalization of prime ideals. An ideal $I$ of $R$ is a {\it $2$-absorbing ideal} if whenever $a,b,c\in R$ and $abc\in I$, then $ab\in I$ or $ac\in I$ or $bc\in I$. In this case $\sqrt{I}=P$ is a prime ideal with $P^2\subseteq I$ or $\sqrt{I}=P_1\cap P_2$ where $P_1,P_2$ are incomparable prime ideals with $P_1P_2\subseteq I$, cf. \cite[Theorem 2.4]{B}. In \cite{AB} Anderson and Badawi introduced the concept of $n$-absorbing ideals as a generalization of prime ideals where $n$ is a positive integer. An ideal $I$ of $R$ is called an {\it $n$-absorbing ideal} of $R$, if whenever $a_1,a_2,\dots,a_{n+1}\in R$ and $\prod_{i=1}^{n+1} a_i\in I$, then there are $n$ of the $a_i$'s whose product is in $I$. In this case, due to Choi and Walker \cite[Theorem 1]{CW}, $(\sqrt {I})^n\subseteq I$.

In \cite{MAD} M. Mukhtar et al. studied the commutative rings whose ideals have a $TA$-factorization. A proper ideal is called a {\it $TA$-ideal} if it is a $2$-absorbing ideal. By a {\it $TA$-factorization} of a proper ideal $I$ we mean an expression of $I$ as a product $\prod_{i=1}^r J_i$ of $TA$-ideals. M. Mukhtar et al. prove that any $TAF$-ring has dimension at most one and the local $TAF$-domains are atomic pseudo-valuations domains. Recently in \cite{ADK}, M. T. Ahmed et al. studied commutative rings whose proper ideals have an $n$-absorbing factorization. Let $I$ be a proper ideal of $R$. By an {\it $n$-absorbing factorization} of $I$ we mean an expression of $I$ as a product $\prod_{i=1}^r I_i$ of proper $n$-absorbing ideals of $R$. M. T. Ahmed et al. called $AF$-$\textrm{dim}(R)$ ({\it absorbing factorization dimension}) the minimum positive integer $n$ such that every ideal of $R$ has an $n$-absorbing factorization. If no such $n$ exists, set $AF$-$\textrm{dim}(R)=\infty$. An {\it $FAF$-ring} ({\it finite absorbing factorization ring}) is a ring such that $AF$-$\textrm{dim}(R)<\infty$. Recall that a {\it general $ZPI$-ring} is a ring whose proper ideals can be written as a product of prime ideals. Therefore, $AF$-$\textrm{dim}(R)$ measures, in some sense, how far $R$ is from being a general $ZPI$-ring, cf. \cite[Proposition 3]{ADK}. By $\dim(R)$ we denote the {\it Krull dimension} of $R$.

In \cite{YNN}, A. Yassine et al. introduced the concept of a $1$-absorbing prime ideal which is a generalization of a prime ideal. A proper ideal $I$ of $R$ is a {\it $1$-absorbing prime ideal} (our abbreviation {\it $OA$-ideal}) if whenever we take nonunit elements $a,b,c\in R$ with $abc\in I$, then $ab\in I$ or $c\in I$. In this case $\sqrt{I}=P$ is a prime ideal, cf. \cite[Theorem 2.3]{YNN}. And if $R$ is a ring in which exists an $OA$-ideal that is not prime, then $R$ is a local ring, that is a ring with one maximal ideal.

Let $I$ be a proper ideal of $R$. By an {\it $OA$-factorization} of $I$ we mean an expression of $I$ as a product $\prod_{i=1}^n J_i$ of $OA$-ideals. The aim of this note is to study the commutative rings whose proper ideals (resp., proper principal ideals, resp., proper $2$-generated ideals) have an $OA$-factorization.

We call $R$ a {\it $1$-absorbing prime factorization ring} ({\it $OAF$-ring}) if every proper ideal has an $OA$-factorization. An {\it $OAF$-domain} is a domain which is an $OAF$-ring. Our paper consists of five sections (including the introduction).

In the next section, we characterize $OA$-ideals (Lemma~\ref{lem1}) and we prove that if $I$ is an $OA$-ideal, then $I$ is a primary ideal. We also show that the $OAF$-ring property is stable under factor ring (resp., fraction ring) formation (Propositions~\ref{prop1} and~\ref{prop2}). Furthermore, we investigate $OAF$-rings with respect to direct products (Corollary~\ref{cor1}) and polynomial ring extensions (Corollary~\ref{cor2}). We prove that the general $ZPI$-rings are exactly the arithmetical $OAF$-rings (Theorem~\ref{thm1}).

The third section consists of a collection of preparational results which will be of major importance in the fourth section. For instance, we show that the Krull dimension of an $OAF$-ring is at most one (Theorem~\ref{thm2}).

The fourth section contains the main results of our paper. Among other results, we provide characterizations of $OAF$-rings (Theorem~\ref{thm4}), rings whose proper principal ideals have an $OA$-factorization (Corollary~\ref{cor3}) and rings whose proper (principal) ideals are $OA$-ideals (Proposition~\ref{prop6}).

In the last section, we study the transfer of the various $OA$-factorization properties to the trivial ring extension.

\section{Characterization of $OA$-ideals and simple facts}

We start with a characterization of $OA$-ideals. Recall that a ring $R$ is a {\it $Q$-ring} (cf. \cite{AM}) if every proper ideal of $R$ is a product of primary ideals.

\begin{lemma}\label{lem1} Let $R$ be a ring with Jacobson radical $M$ and $I$ be an ideal of $R$.
\begin{enumerate}
\item[\textnormal{(1)}] If $R$ is not local, then $I$ is an $OA$-ideal if and only if $I$ is a prime ideal.
\item[\textnormal{(2)}] If $R$ is local, then $I$ is an $OA$-ideal if and only if $I$ is a prime ideal or $M^2\subseteq I\subseteq M$.
\item[\textnormal{(3)}] Every $OA$-ideal is a primary $TA$-ideal. In particular, every $OAF$-ring is both a $Q$-ring and a $TAF$-ring.
\end{enumerate}
\end{lemma}

\begin{proof} (1) This follows from \cite[Theorem 2.4]{YNN}.

(2) Let $R$ be local. Then $M$ is the maximal ideal of $R$.

($\Rightarrow$) Let $I$ be an $OA$-ideal such that $I$ is not a prime ideal. Since $I$ is proper, we infer that $I\subseteq M$. Since $I$ is not prime, there are $a,b\in M\setminus I$ such that $ab\in I$. To prove that $M^2\subseteq I$, it suffices to show that $xy\in I$ for all $x,y\in M$. Let $x,y\in M$. Then $xyab\in I$. Since $xy,a,b\in M$, $b\not\in I$ and $I$ is an $OA$-ideal, it follows that $xya\in I$. Again, since $x,y,a\in M$, $a\not\in I$ and $I$ is an $OA$-ideal, we have that $xy\in I$.

($\Leftarrow$) Clearly, if $I$ is a prime ideal, then $I$ is an $OA$-ideal. Now let $M^2\subseteq I\subseteq M$. Then $I$ is proper. Let $a,b,c\in M$ be such that $abc\in I$. Then $ab\in M^2\subseteq I$. Therefore, $I$ is an $OA$-ideal.

(3) Let $I$ be an $OA$-ideal. It is an immediate consequence of (1) and (2) that $I$ is a primary ideal. Now let $a,b,c\in R$ be such that $abc\in I$. We have to show that $ab\in I$ or $ac\in I$ or $bc\in I$.

First let $a$ or $b$ or $c$ be a unit of $R$. Without restriction let $a$ be unit of $R$. Since $abc\in I$, we infer that $bc\in I$.

Now let $a$, $b$ and $c$ be nonunits. Then $ab\in I$ or $c\in I$. If $c\in I$, then $ac\in I$. The in particular statement is clear.
\end{proof}

\begin{proposition}\label{prop1} Let $R$ be an $OAF$-ring and $I$ be a proper ideal of $R$. Then $R/I$ is an $OAF$-ring.
\end{proposition}

\begin{proof} Let $J$ be a proper ideal of $R$ which contains $I$. Let $J=\prod_{i=1}^m J_i$ be an $OA$-factorization. Then $J/I=\prod_{i=1}^m (J_i/I)$. It suffices to show that $J_i/I$ is an $OA$-ideal for each $i\in [1,m]$. Let $i\in [1,m]$ and let $a,b,c\in R$ be such that $\bar{a},\bar{b},\bar{c}$ are three nonunit elements of $R/I$ and $\bar{a}\bar{b}\bar{c}\in J_i/I$. Clearly, $a,b,c$ are nonunit elements of $R$ and $abc\in J_i$. Since $J_i$ is an $OA$-ideal of $R$, we get that $ab\in J_i$ or $c\in J_i$ which implies that $\bar{a}\bar{b}\in J_i/I$ or $\bar{c}\in J_i/I$. Therefore, $R/I$ is an $OAF$-ring.
\end{proof}

\begin{proposition}\label{prop2} Let $S$ be a multiplicatively closed subset of $R\setminus\boldsymbol{0}$. If $R$ is an $OAF$-ring, then $S^{-1}R$ is an $OAF$-ring. In particular, $R_M$ is an $OAF$-ring for every maximal ideal $M$ of $R$.
\end{proposition}

\begin{proof} Let $J$ be a proper ideal of $S^{-1}R$. Then $J=S^{-1}I$ for some proper ideal $I$ of $R$ with $I\cap S=\varnothing$. Let $I=\prod_{i=1}^m I_i$ be an $OA$-factorization. Then $J=\prod_{i=1}^m (S^{-1}I_i)$ where each $S^{-1}I_i$ which is proper is an $OA$-ideal by \cite[Theorem 2.18]{YNN}. Thus $S^{-1}R$ is an $OAF$-ring. The in particular statement is clear.
\end{proof}

Let $R$ be a ring. Then $R$ is said to be a {\it $\pi$-ring} if every proper principal ideal of $R$ is a product of prime ideals. We say that $R$ is a {\it unique factorization ring} (in the sense of Fletcher, cf. \cite{AMA}) if every proper principal ideal of $R$ is a product of principal prime ideals. A {\it unique factorization domain} is an integral domain which is a unique factorization ring.

\begin{remark}\label{rem1} Let $R$ be a non local ring.
\begin{enumerate}
\item[\textnormal{(1)}] $R$ is a general $ZPI$-ring if and only if $R$ is an $OAF$-ring.
\item[\textnormal{(2)}] $R$ is a $\pi$-ring if and only if each proper principal ideal of $R$ has an $OA$-factorization.
\item[\textnormal{(3)}] $R$ is a unique factorization ring if and only if each proper principal ideal of $R$ is a product of principal $OA$-ideals.
\end{enumerate}
\end{remark}

\begin{proof} This is an immediate consequence of Lemma~\ref{lem1}(1).
\end{proof}

In the light of the above remark we give the next result.

\begin{corollary}\label{cor1}
Let $R_1$ and $R_2$ be two rings and $R=R_1\times R_2$ be their direct product. The following statements are equivalent.
\begin{enumerate}
\item[\textnormal{(1)}] $R$ is an $OAF$-ring.
\item[\textnormal{(2)}] $R$ is a general $ZPI$-ring.
\item[\textnormal{(3)}] $R_1$ and $R_2$ are general $ZPI$-rings.
\end{enumerate}
\end{corollary}

\begin{proof} This follows from Remark~\ref{rem1}(1) and \cite[Exercise 6(g), page 223]{LM}.
\end{proof}

Let $R$ be a ring. Then $R$ is called a {\it von Neumann regular ring} if for each $x\in R$ there is some $y\in R$ with $x=x^2y$. The ring $R$ is von Neumann regular if and only if $R$ is a zero-dimensional reduced ring (see \cite[Theorem 3.1, page 10]{H}).

\begin{corollary}\label{cor2} Let $R$ be a ring. The following statements are equivalent.
\begin{enumerate}
\item[\textnormal{(1)}] $R[X]$ is an $OAF$-ring.
\item[\textnormal{(2)}] $R$ is a Noetherian von Neumann regular ring.
\item[\textnormal{(3)}] $R$ is a finite direct product of fields.
\end{enumerate}
\end{corollary}

\begin{proof} Observe that the polynomial ring $R[X]$ is never local, since $X$ and $1-X$ are nonunit elements of $R[X]$, but their sum is a unit. Consequently, $R[X]$ is an $OAF$-ring if and only if $R[X]$ is a general $ZPI$-ring by Remark~\ref{rem1}(1). The rest is now an easy consequence of \cite[Theorem 6 and Corollary 6.1]{A}, \cite[Exercise 10, page 225]{LM} and Hilbert's basis theorem.
\end{proof}

Let $R$ be a ring and $I$ be an ideal of $R$. Then $I$ is called {\it divided} if $I$ is comparable to every ideal of $R$ (or equivalently, $I$ is comparable to every principal ideal of $R$).

\begin{lemma}\label{lem2} Let $R$ be a local ring with maximal ideal $M$ such that $M^2$ is divided. The following statements are equivalent.
\begin{enumerate}
\item[\textnormal{(1)}] Each two principal $OA$-ideals which contain $M^2$ are comparable.
\item[\textnormal{(2)}] For each $OA$-ideal $I$ of $R$, we have that $I$ is a prime ideal or $I=M^2$.
\end{enumerate}
\end{lemma}

\begin{proof} (1) $\Rightarrow$ (2): Let $I$ be an $OA$-ideal of $R$ such that $I$ is not a prime ideal of $R$. Then $M^2\subseteq I\subset M$ by Lemma~\ref{lem1}(2). Assume that $M^2\subset I$. Let $x\in I\setminus M^2$ and let $y\in M\setminus I$. Then $x,y\not\in M^2$, and thus $M^2\subseteq xR,yR$ (since $M^2$ is divided). It follows that $xR$ and $yR$ are (principal) $OA$-ideals of $R$ by Lemma~\ref{lem1}(2). Since $y\not\in xR$ and $xR$ and $yR$ are comparable, we infer that $xR\subset yR$. Consequently, there is some $z\in M$ such that $x=yz$, and hence $x\in M^2$, a contradiction. Therefore, $I=M^2$.

(2) $\Rightarrow$ (1): This is obvious.
\end{proof}

Let $R$ be a ring. An ideal $I$ of $R$ is called {\it $2$-generated} if $I=xR+yR$ for some (not necessarily distinct) $x,y\in R$. Note that every principal ideal of $R$ is $2$-generated. We say that $R$ is a {\it chained ring} if each two ideals of $R$ are comparable under inclusion. Moreover, $R$ is said to be an {\it arithmetical ring} if $R_M$ is a chained ring for each maximal ideal $M$ of $R$.

\begin{theorem}\label{thm1} Let $R$ be a ring. The following statements are equivalent.
\begin{enumerate}
\item[\textnormal{(1)}] $R$ is a general $ZPI$-ring
\item[\textnormal{(2)}] $R$ is an arithmetical $OAF$-ring.
\item[\textnormal{(3)}] $R$ is an arithmetical ring and each proper principal ideal of $R$ has an $OA$-factorization.
\end{enumerate}
\end{theorem}

\begin{proof} First we show that if $R$ is an arithmetical $\pi$-ring, then $R$ is a general $ZPI$-ring. Let $R$ be an arithmetical $\pi$-ring and let $M$ be a maximal ideal of $R$. It is straightforward to show that $R_M$ is a $\pi$-ring. Moreover, $R_M$ is a chained ring, and hence every $2$-generated ideal of $R_M$ is principal. Therefore, every proper $2$-generated ideal of $R_M$ is a product of prime ideals of $R_M$. Consequently, $R_M$ is a general $ZPI$-ring by \cite[Theorem 3.2]{L}. This implies that $\dim(R_M)\leq 1$ by \cite[page 205]{LM}. We infer that $\dim(R)\leq 1$, and thus $R$ is a general $ZPI$-ring by \cite[Theorems 39.2, 46.7, and 46.11]{GI}.

(1) $\Rightarrow$ (2) $\Rightarrow$ (3): This is obvious.

(3) $\Rightarrow$ (1): It is sufficient to show that $R$ is a $\pi$-ring. If $R$ is not local, then $R$ is a $\pi$-ring by Remark~\ref{rem1}(2). Therefore, we can assume that $R$ is local with maximal ideal $M$. Since $R$ is local, we have that $R$ is a chained ring. Therefore, $M^2$ is divided and each two $OA$-ideals of $R$ are comparable. We infer by Lemma~\ref{lem2} that each $OA$-ideal of $R$ is a product of prime ideals. Now it clearly follows that $R$ is a $\pi$-ring.
\end{proof}

\section{Preparational results}

From Lemma~\ref{lem1}(3), we have that $|{\rm Min}(I)|=1$ for every $OA$-ideal $I$ of $R$. In view of this remark, we obtain the following result.

\begin{proposition}\label{prop3} Let $R$ be a ring and $I$ be a proper ideal of $R$. If $I$ has an $OA$-factorization, then ${\rm Min}(I)$ is finite.
\end{proposition}

\begin{proof} Let $I=\prod_{i=1}^n I_i$ be an $OA$-factorization. It follows that ${\rm Min}(I)\subseteq\bigcup_{i=1}^n {\rm Min}(I_i)$, and thus $|{\rm Min}(I)|\leq n$.
\end{proof}

Let $R$ be a ring and $I$ be an ideal of $R$. Then $I$ is called a {\it multiplication ideal} of $R$ if for each ideal $J$ of $R$ with $J\subseteq I$, there is some ideal $L$ of $R$ such that $J=IL$.

\begin{lemma}\label{lem3} Let $R$ be a local ring such that each proper principal ideal of $R$ has an $OA$-factorization. Then each nonmaximal minimal prime ideal of $R$ is principal.
\end{lemma}

\begin{proof} Let $P$ be a nonmaximal minimal prime ideal of $R$. By \cite[Theorem 1]{A} it is sufficient to show that $P$ is a multiplication ideal.

Let $x\in P$ and let $xR=\prod_{i=1}^n I_i$ be an $OA$-factorization. There is some $j\in [1,n]$ such that $I_j\subseteq P$. By Lemma~\ref{lem1}(2) we have that $P=I_j$, and hence $xR=PJ$ for some ideal $J$ of $R$. We infer that $xR=P(xR:P)$.

Now let $I$ be an ideal of $R$ such that $I\subseteq P$. Then $I=\sum_{y\in I} yR=\sum_{y\in I} P(yR:P)=P\sum_{y\in I} (yR:P)$, and thus $P$ is a multiplication ideal.
\end{proof}

The next result is a generalization of \cite[Theorem 46.8]{GI} and its proof is based on the proof of the same result.

\begin{proposition}\label{prop4} Let $R$ be a local ring with maximal ideal $M$ such that $\dim(R)\geq 1$ and every proper principal ideal of $R$ has an $OA$-factorization. Then $R$ is an integral domain and if $\dim(R)\geq 2$, then $R$ is a unique factorization domain.
\end{proposition}

\begin{proof} Let $N$ be the nilradical of $R$. It follows from Proposition~\ref{prop3} and Lemma~\ref{lem3} that ${\rm Min}(\boldsymbol{0})$ is finite and each $P\in {\rm Min}(\boldsymbol{0})$ is principal.

\textsc{Claim}: Every proper principal ideal of $R/N$ has an $OA$-factorization. Let $I$ be a proper principal ideal of $R/N$. Then $I=(xR+N)/N$ for some $x\in M$. Let $xR=\prod_{i=1}^n I_i$ be an $OA$-factorization. We infer that $I=(xR)/N=(\prod_{i=1}^n I_i)/N=\prod_{i=1}^n (I_i/N)$. It suffices to show that $I_i/N$ is an $OA$-ideal of $R/N$ for each $i\in [1,n]$. Let $i\in [1,n]$. If $I_i$ is a prime ideal of $R$, then $N\subseteq I_i$, and hence $I_i/N$ is a prime ideal of $R/N$. Now let $I_i$ be not a prime ideal of $R$. By Lemma~\ref{lem1}(2), we have that $M^2\subseteq I_i\subseteq M$. Note that $R/N$ is local with maximal ideal $M/N$. Since $(M/N)^2=M^2/N\subseteq I_i/N\subseteq M/N$, it follows by Lemma~\ref{lem1}(2) that $I_i/N$ is an $OA$-ideal of $R/N$. This proves the claim.

\textsc{Case} 1: $R$ is one-dimensional. We prove that $R$ is an integral domain. If every $OA$-ideal of $R$ is a prime ideal, then $R$ is $\pi$-ring, and hence $R$ is an integral domain by \cite[Theorem 46.8]{GI}. Now let not every $OA$-ideal of $R$ be a prime ideal. It follows from Lemma~\ref{lem1}(2) that $M$ is not idempotent. Set $L=M^2\cup\bigcup_{Q\in {\rm Min}(\boldsymbol{0})} Q$. Next we prove that $M^2\subseteq xR$ for each $x\in R\setminus L$. Let $x\in R\setminus L$. Without restriction let $x$ be a nonunit. Note that $xR$ cannot be a product of more than one $OA$-ideal, and hence $xR$ is an $OA$-ideal. By Lemma~\ref{lem1}(2) we have that $M^2\subseteq xR$.

Now we show that $P\subseteq M^2$ for each $P\in {\rm Min}(\boldsymbol{0})$. Let $P\in {\rm Min}(\boldsymbol{0})$. Assume that $P\nsubseteq M^2$. Let $w\in R\setminus P$. Then $P+wR\nsubseteq L$ by the prime avoidance lemma, and thus there is some $v\in (P+wR)\setminus L$. It follows that $M^2\subseteq vR\subseteq P+wR$. Since $P$ is a nonmaximal prime ideal, we have that $R/P$ has no simple $R/P$-submodules, and hence $\bigcap_{y\in R\setminus P} (P+yR)=P$. (Note that if $\bigcap_{y\in R\setminus P} (P+yR)\not=P$, then $\bigcap_{y\in R\setminus P} (P+yR)/P$ is a simple $R/P$-submodule of $R/P$.) This implies that $M^2\subseteq\bigcap_{y\in R\setminus P} (P+yR)=P$, and thus $P=M$, a contradiction.

Let $Q\in {\rm Min}(\boldsymbol{0})$. By the prime avoidance lemma, there is some $z\in M\setminus L$. We infer that $Q\subset M^2\subset zR$. Consequently, $Q=zQ$. Since $Q$ is principal, it follows that $Q=\boldsymbol{0}$ (e.g. by Nakayama's lemma), and hence $R$ is an integral domain.

\textsc{Case} 2: $\dim(R)\geq 2$ and $R$ is reduced. We show that $R$ is a unique factorization domain. There is some nonmaximal nonminimal prime ideal $Q$ of $R$. By the prime avoidance lemma, there is some $x\in Q\setminus\bigcup_{P\in {\rm Min}(\boldsymbol{0})} P$. Since $R$ is reduced, we have that $\bigcap_{L\in {\rm Min}(\boldsymbol{0})} L=\boldsymbol{0}$. If $y\in R$ is nonzero with $xy=0$, then $y\not\in L$ and $xy\in L$ for some $L\in {\rm Min}(\boldsymbol{0})$, and hence $x\in L$, a contradiction. We infer that $x$ is a regular element of $R$. Let $xR=\prod_{i=1}^n I_i$ be an $OA$-factorization. Then $I_j\subseteq Q$ for some $j\in [1,n]$. Since $x$ is regular, $I_j$ is invertible, and hence $I_j$ is a regular principal ideal (because invertible ideals of a local ring are regular principal ideals). Since $I_j\subseteq Q$ and $Q\not=M$, we have that $I_j$ is a prime ideal by Lemma~\ref{lem1}(2). Consequently, $P\subseteq I_j$ for some $P\in {\rm Min}(\boldsymbol{0})$. Since $I_j$ is regular, we infer that $P\subset I_j$, and hence $P=PI_j$ (since $I_j$ is principal). It follows (e.g. from Nakayama's lemma) that $P=\boldsymbol{0}$ (since $P$ is principal). We obtain that $R$ is an integral domain.

To show that $R$ is a unique factorization domain, it suffices to show by \cite[Theorem 2.6]{AMA} that every nonzero prime ideal of $R$ contains a nonzero principal prime ideal. Since $\dim(R)\geq 2$ and $R$ is local, we only need to show that every nonzero nonmaximal prime ideal of $R$ contains a nonzero principal prime ideal. Let $L$ be a nonzero nonmaximal prime ideal of $R$ and let $z\in L$ be nonzero. Let $zR=\prod_{k=1}^m J_k$ be an $OA$-factorization. Then $J_{\ell}\subseteq L$ for some $\ell\in [1,m]$. Since $R$ is an integral domain, $zR$ is invertible, and hence $J_{\ell}$ is invertible. Therefore, $J_{\ell}$ is nonzero and principal (since $R$ is local). Since $L\not=M$, it follows from Lemma~\ref{lem1}(2) that $J_{\ell}$ is a prime ideal.

\textsc{Case} 3: $\dim(R)\geq 2$. We have to show that $R$ is a unique factorization domain. Note that $R/N$ is a reduced local ring with maximal ideal $M/N$ and $\dim(R/N)\geq 2$. Moreover, each proper principal ideal of $R/N$ has an $OA$-factorization by the claim. It follows by Case 2 that $R/N$ is a unique factorization domain, and thus $N$ is the unique minimal prime ideal of $R$. Since $R/N$ is a unique factorization domain and $\dim(R/N)\geq 2$, $R/N$ possesses a nonzero nonmaximal principal prime ideal. We infer that there is some nonminimal nonmaximal prime ideal $Q$ of $R$ such that $Q/N$ is a principal ideal of $R/N$. Consequently, there is some $q\in Q$ such that $Q=qR+N$. Let $qR=\prod_{i=1}^n I_i$ be an $OA$-factorization. Then $I_j\subseteq Q$ for some $j\in [1,n]$. Since $Q\not=M$, we infer by Lemma~\ref{lem1}(2) that $I_j$ is a prime ideal of $R$. Therefore, $Q=qR+N\subseteq I_j\subseteq Q$, and hence $I_j=Q$.

Assume that $Q\not=qR$. Then $qR=QJ$ for some proper ideal $J$ of $R$. It follows that $q\in qR=(qR+N)J\subseteq qJ+N$, and thus $q(1-a)\in N$ for some $a\in J$. Since $a$ is a nonunit of $R$, we obtain that $q\in N$. This implies that $Q=qR+N=N$, a contradiction. We infer that $Q=qR$. Since $N\subset Q$ and $N$ is a prime ideal of $R$, we have that $N=NQ$. Consequently, $N=\boldsymbol{0}$ (e.g. by Nakayama's lemma, since $N$ is principal), and thus $R\cong R/N$ is a unique factorization domain.
\end{proof}

\begin{proposition}\label{prop5} Let $R$ be a local ring with maximal ideal $M$ such that each proper $2$-generated ideal of $R$ has an $OA$-factorization. Then $\dim(R)\leq 2$ and each nonmaximal prime ideal of $R$ is principal.
\end{proposition}

\begin{proof} First we show that $\dim(R_P)\leq 1$ for each nonmaximal prime ideal $P$ of $R$. Let $P$ be a nonmaximal prime ideal and let $I$ be a proper $2$-generated ideal of $R_P$. Observe that $I=J_P$ for some $2$-generated ideal $J$ of $R$ with $J\subseteq P$. Let $J=\prod_{i=1}^n J_i$ be an $OA$-factorization. Then $I=J_P=\prod_{i=1}^n (J_i)_P=\prod_{i=1,J_i\subseteq P}^n (J_i)_P$. If $i\in [1,n]$ is such that $J_i\subseteq P$, then $J_i$ is a prime ideal of $R$ by Lemma~\ref{lem1}(2), and thus $(J_i)_P$ is a prime ideal of $R_P$. We infer that $I$ is a product of prime ideals of $R_P$. It follows from \cite[Theorem 3.2]{L}, that $R_P$ is a general $ZPI$-ring. It is an easy consequence of \cite[page 205]{LM} that $\dim(R_P)\leq 1$.

This implies that $\dim(R)\leq 2$. It remains to show that every nonmaximal prime ideal of $R$ is principal. Without restriction let $\dim(R)\geq 1$. It follows from Proposition~\ref{prop4} that $R$ is either a one-dimensional domain or a two-dimensional unique factorization domain. In any case we have that each nonmaximal prime ideal of $R$ is principal.
\end{proof}

In the next result we will prove a generalization of the fact that every $OAF$-ring has Krull dimension at most one.

\begin{theorem}\label{thm2} Let $R$ be a ring such that every proper $2$-generated ideal of $R$ has an $OA$-factorization. Then $\dim(R)\leq 1$.
\end{theorem}

\begin{proof} If every $OA$-ideal of $R$ is a prime ideal, then $R$ is a general $ZPI$-ring by \cite[Theorem 3.2]{L}, and hence $\dim(R)\leq 1$ by \cite[page 205]{LM}. Now let not every $OA$-ideal of $R$ be a prime ideal. We infer by Lemma~\ref{lem1} that $R$ is local and the maximal ideal of $R$ is not idempotent. Let $M$ be the maximal ideal of $R$. It suffices to show that if $Q$ is a nonmaximal prime ideal of $R$, then $Q=\boldsymbol{0}$. Let $Q$ be a nonmaximal prime ideal of $R$.

Assume that $Q\nsubseteq M^2$. Since $\dim(R)\leq 2$ by Proposition~\ref{prop5}, there is some prime ideal $P$ of $R$ such that $Q\subseteq P$ and $\dim(R/P)=1$. Next we show that $M^2\subseteq P+yR$ for each $y\in R\setminus P$. Let $y\in R\setminus P$ and set $J=P+yR$. Without restriction let $J\subset M$. Note that $J$ is $2$-generated by Proposition~\ref{prop5}. Since $J\nsubseteq M^2$, $J$ cannot be a product of more than one $OA$-ideal, and thus $J$ is an $OA$-ideal of $R$. Since $P\subset J\subset M$, we have that $J$ is not a prime ideal of $R$, and thus $M^2\subseteq J$ by Lemma~\ref{lem1}(2). Moreover, $R/P$ is an integral domain that is not a field. Consequently, $R/P$ does not have any simple $R/P$-submodules, which implies that $P=\bigcap_{x\in R\setminus P} (P+xR)$. (Observe that if $\bigcap_{x\in R\setminus P} (P+xR)\not=P$, then $\bigcap_{x\in R\setminus P} (P+xR)/P$ is a simple $R/P$-submodule of $R/P$.) Therefore, $M^2\subseteq\bigcap_{x\in R\setminus P} (P+xR)=P$, and hence $P=M$, a contradiction. We infer that $Q\subseteq M^2$.

There is some $z\in M\setminus M^2$ (since $M$ is not idempotent). Since $zR$ is a product of $OA$-ideals, we have that $zR$ is an $OA$-ideal of $R$. As shown before, $L\subseteq M^2$ for each nonmaximal prime ideal $L$ of $R$, and thus $zR$ is not a nonmaximal prime ideal. Consequently, $Q\subset M^2\subset zR$ by Lemma~\ref{lem1}(2), and hence $Q=zQ$. Since $Q$ is principal by Proposition~\ref{prop5}, it follows (e.g. by Nakayama's lemma) that $Q=\boldsymbol{0}$.
\end{proof}

\begin{lemma}\label{lem4} Let $D$ be a local domain with maximal ideal $M$. Then each proper principal ideal of $D$ has an $OA$-factorization if and only if $D$ is atomic and each irreducible element generates an $OA$-ideal. If these equivalent conditions are satisfied, then $\bigcap_{n\in\mathbb{N}} P^n=\boldsymbol{0}$ for each height-one prime ideal $P$ of $D$.
\end{lemma}

\begin{proof} ($\Rightarrow$) Let each proper principal ideal of $D$ have an $OA$-factorization. If $D$ is a unique factorization domain, then $D$ is atomic and each irreducible element generates a prime ideal. Now let $D$ be not a unique factorization domain. Then $\dim(D)=1$ by Proposition~\ref{prop4}.

Assume that $M^2$ is principal. Then $M$ is invertible, and hence $M$ is principal (since $D$ is local). Note that $D$ is a $DVR$ (since $\dim(D)=1$), and hence $D$ is a unique factorization domain, a contradiction.

We infer that $M^2$ is not principal. We show that $D$ is atomic. Let $y\in D$ be a nonzero nonunit. Then $yD=\prod_{i=1}^n I_i$ for some principal $OA$-ideals $I_i$. There are nonzero nonunits $x_i\in D$ such that $y=\prod_{i=1}^n x_i$ and $I_j=x_jD$ for each $j\in [1,n]$. Let $i\in [1,n]$. If $I_i$ is a prime ideal, then $x_i$ is a prime element, and thus $x_i$ is irreducible. Now let $I_i$ not be a prime ideal. It follows from Lemma~\ref{lem1}(2) that $M^2\subseteq I_i$. Since $M^2$ is not principal, we have that $x_i\not\in M^2$. Therefore, $x_i$ is irreducible.

Finally, let $z\in D$ be irreducible. Then $zD=\prod_{j=1}^m J_j$ for some principal $OA$-ideals $J_j$. Since $zD$ is maximal among the proper principal ideals of $D$, we obtain that $zD=J_j$ for some $j\in [1,n]$.

($\Leftarrow$) Let $D$ be atomic such that each irreducible element generates an $OA$-ideal. Let $I$ be a proper principal ideal of $D$. Without restriction let $I$ be nonzero. Then $I=xD$ for some nonzero nonunit $x\in D$. Observe that $x=\prod_{i=1}^n x_i$ for some irreducible elements $x_i\in D$. It follows that $\prod_{i=1}^n x_iD$ is an $OA$-factorization of $I$.

Now let the equivalent conditions be satisfied and let $P$ be a height-one prime ideal of $D$. First let $P\not=M$. Then $D$ is a unique factorization domain by Proposition~\ref{prop4}, and hence $P$ is principal. Therefore, $\bigcap_{n\in\mathbb{N}} P^n$ is a prime ideal of $D$ by \cite[Theorem 2.2(1)]{AMN}. Since $\bigcap_{n\in\mathbb{N}} P^n\subset P$, we infer that $\bigcap_{n\in\mathbb{N}} P^n=\boldsymbol{0}$.

Now let $P=M$. Assume that $\bigcap_{n\in\mathbb{N}} M^n\not=\boldsymbol{0}$ and let $x\in\bigcap_{n\in\mathbb{N}} M^n$ be nonzero. Then $xD$ is a product of $m$ $OA$-ideals of $D$ for some positive integer $m$. We infer by Lemma~\ref{lem1}(2) that $M^{2m}\subseteq xD$, and hence $M^{2m}\subseteq xD\subseteq M^{4m}\subseteq M^{2m}$. This implies that $xD=M^{2m}=M^{4m}=x^2D$, and thus $x$ is a unit of $D$, a contradiction. Therefore, $\bigcap_{n\in\mathbb{N}} M^n=\boldsymbol{0}$.
\end{proof}

\begin{lemma}\label{lem5} Let $R$ be a local ring with maximal ideal $M$ such that $M^2$ is divided and such that either $M$ is nilpotent or $R$ is an integral domain with $\bigcap_{n\in\mathbb{N}} M^n=\boldsymbol{0}$. Then $R$ is an $OAF$-ring and every proper principal ideal of $R$ is a product of principal $OA$-ideals.
\end{lemma}

\begin{proof} If $M$ is idempotent, then $M=\boldsymbol{0}$, and hence $R$ is a field and both statements are clearly satisfied. Now let $M$ be not idempotent. There is some $x\in M\setminus M^2$. In what follows, we freely use the fact that if $N$ is an ideal of $R$ and $z\in R$ such that $N\subseteq zR$, then $N=z(N:zR)$, and hence $N=zJ$ for some ideal $J$ of $R$.

Next we prove that $M^2=xM$ and $xR$ is an $OA$-ideal of $R$. Since $x\not\in M^2$ and $M^2$ is divided, we have that $M^2\subseteq xR\subseteq M$. Therefore, $xR$ is an $OA$-ideal by Lemma~\ref{lem1}(2). Since $M^2\subset xR$, there is some proper ideal $J$ of $R$ with $M^2=xJ$, and thus $M^2\subseteq xM$. Obviously, $xM\subseteq M^2$, and hence $M^2=xM$.

Now we show that $R$ is an $OAF$-ring. Let $I$ be a proper ideal of $R$. First let $I=\boldsymbol{0}$. If $M$ is nilpotent, then $I$ is obviously a product of $OA$-ideals. If $R$ is an integral domain, then $I$ is an $OA$-ideal. Now let $I$ be nonzero. In any case there is a largest positive integer $n$ such that $I\subseteq M^n$. Observe that $I\subseteq M^n=x^{n-1}M\subseteq x^{n-1}R$. Consequently, $I=x^{n-1}L=(xR)^{n-1}L$ for some proper ideal $L$ of $R$. Assume that $L\subseteq M^2$. Note that $L\subseteq M^2=xM\subseteq xR$. This implies that $L=xA$ for some proper ideal $A$ of $R$, and hence $I=x^nA\subseteq x^nM=M^{n+1}$, a contradiction. We infer that $M^2\subseteq L$ (since $M^2$ is divided). It follows from Lemma~\ref{lem1}(2) that $L$ is an $OA$-ideal. In any case, $I$ is a product of $OA$-ideals.

Finally, we prove that every proper principal ideal of $R$ is a product of principal $OA$-ideals. Let $y\in M$. First let $y=0$. If $M$ is nilpotent, then $x^k=0$ for some $k\in\mathbb{N}$, and thus $yR=(xR)^k$ is a product of principal $OA$-ideals. If $R$ is an integral domain, then $yR$ is a principal $OA$-ideal. Now let $y$ be nonzero. There is some greatest $\ell\in\mathbb{N}$ such that $y\in M^{\ell}$. Therefore, $y=x^{\ell-1}z$ for some $z\in M$. If $z\in M^2$, then $z=xv$ for some $v\in M$, and hence $y=x^{\ell}v\in M^{\ell+1}$, a contradiction. We infer that $z\not\in M^2$, and thus $M^2\subseteq zR\subseteq M$. It follows from Lemma~\ref{lem1}(2) that $zR$ is an $OA$-ideal of $R$. Consequently, $yR=(xR)^{\ell-1}(zR)$ is a product of principal $OA$-ideals.
\end{proof}

\section{Characterization of $OAF$-rings and related concepts}

First we recall several definitions and discuss the factorization theoretical properties of local one-dimensional $OAF$-domains. Let $D$ be an integral domain with quotient field $K$. Then $\widehat{D}=\{x\in K\mid$ there is some nonzero $c\in D$ such that $cx^n\in D$ for all $n\in\mathbb{N}\}$ is called the {\it complete integral closure} of $D$. Let $(D:\widehat{D})=\{x\in D\mid x\widehat{D}\subseteq D\}$ be the {\it conductor} of $D$ in $\widehat{D}$. The domain $D$ is called {\it completely integrally closed} if $D=\widehat{D}$ and $D$ is said to be {\it seminormal} if for all $x\in K$ such that $x^2,x^3\in D$, it follows that $x\in D$. Note that every completely integrally closed domain is seminormal. We say that $D$ is a {\it finitely primary domain of rank one} if $D$ is a local one-dimensional domain such that $\widehat{D}$ is a $DVR$ and $(D:\widehat{D})\not=\boldsymbol{0}$. For each subset $X\subseteq K$ let $X^{-1}=\{x\in K\mid xX\subseteq D\}$ and $X_v=(X^{-1})^{-1}$. An ideal $I$ of $D$ is called {\it divisorial} if $I_v=I$. Moreover, $D$ is called a {\it Mori domain} if $D$ satisfies the ascending chain condition on divisorial ideals. It is well known that every unique factorization domain and every Noetherian domain is a Mori domain (see \cite[Corollary 2.3.13]{GH} and \cite[page 57]{BA}). We say that $D$ is {\it half-factorial} if $D$ is atomic and each two factorizations of each nonzero element of $D$ into irreducible elements are of the same length. Finally, $D$ is called a {\it $C$-domain} if the monoid of nonzero elements of $D$ (i.e., $D\setminus\boldsymbol{0}$) is a $C$-monoid. For the precise definition of $C$-monoids we refer to \cite[Definition 2.9.5]{GH}.

Let $D$ be a local domain with quotient field $K$ and maximal ideal $M$. Set $(M:M)=\{x\in K\mid xM\subseteq M\}$. Then $(M:M)$ is called the {\it ring of multipliers} of $M$. Moreover, $M^2$ is said to be {\it universal} if $M^2\subseteq uD$ for each irreducible element $u\in D$.

\begin{theorem}\label{thm3} Let $D$ be a local domain with maximal ideal $M$ such that $D$ is not a field. The following statements are equivalent.
\begin{enumerate}
\item[\textnormal{(1)}] $D$ is an $OAF$-domain.
\item[\textnormal{(2)}] $D$ is a $TAF$-domain.
\item[\textnormal{(3)}] $D$ is one-dimensional and every proper principal ideal has an $OA$-factorization.
\item[\textnormal{(4)}] $D$ is one-dimensional and atomic and every irreducible element generates an $OA$-ideal.
\item[\textnormal{(5)}] $D$ is atomic such that $M^2$ is universal.
\item[\textnormal{(6)}] $(M:M)$ is a $DVR$ with maximal ideal $M$.
\item[\textnormal{(7)}] $D$ is a seminormal finitely primary domain of rank one.
\end{enumerate}
If these equivalent conditions are satisfied, then $D$ is a half-factorial $C$-domain and a Mori domain.
\end{theorem}

\begin{proof} (1) $\Rightarrow$ (2): This follows from Lemma~\ref{lem1}(3).

(1) $\Rightarrow$ (3): By Theorem~\ref{thm2}, $D$ is one-dimensional. The rest of assertion (3) is clear.

(2) $\Leftrightarrow$ (5) $\Leftrightarrow$ (6): This follows from \cite[Theorem 4.3]{MAD}.

(3) $\Leftrightarrow$ (4): This is an immediate consequence of Lemma~\ref{lem4}.

(4) $\Rightarrow$ (5): Let $y\in D$ be an irreducible element. Since $yD$ is an $OA$-ideal and $\sqrt{yD}=M$, we deduce from Lemma~\ref{lem1}(2) that $M^2\subseteq yD$. Hence $M^2$ is universal.

(5)+(6) $\Rightarrow$ (1): It follows from \cite[Theorem 5.1]{AMO} that $M^2$ is comparable to every principal ideal of $D$, and thus $M^2$ is divided. Since $(M:M)$ is a $DVR$ with maximal ideal $M$, we have that $\bigcap_{n\in\mathbb{N}} M^n=\boldsymbol{0}$. Consequently, $D$ is an $OAF$-domain by Lemma~\ref{lem5}.

(5)+(6) $\Rightarrow$ (7): First we show that $D$ is finitely primary of rank one. Let $P$ be a nonzero prime ideal of $D$. Then $P$ contains an irreducible element $y\in D$, and hence $M^2\subseteq yD\subseteq P$. Therefore, $P=M$, and thus $D$ is one-dimensional. It remains to show that $\widehat{D}$ is a $DVR$ and $(D:\widehat{D})\not=\boldsymbol{0}$. Since $(M:M)$ is a $DVR$, we have that $(M:M)$ is completely integrally closed. Observe that $D\subseteq (M:M)\subseteq\widehat{D}$, and hence $\widehat{D}\subseteq\widehat{(M:M)}=(M:M)$. Therefore, $\widehat{D}=(M:M)$ is a $DVR$. Since $M\widehat{D}=M(M:M)\subseteq M\subseteq D$ and $M\not=\boldsymbol{0}$, we infer that $(D:\widehat{D})\not=\boldsymbol{0}$.

Next we show that $D$ is seminormal. Let $V$ be the group of units of $\widehat{D}$. Let $K$ be the field of quotients of $D$ and let $x\in K$ be such that $x^2,x^3\in D$. Then $x^2,x^3\in\widehat{D}$. Since $\widehat{D}$ is a $DVR$, $\widehat{D}$ is seminormal, and thus $x\in\widehat{D}$. In particular, $x\in M$ or $x\in V$. If $x\in M$, then $x\in D$. Now let $x\in V$. Note that $V\cap D$ is the group of units of $D$ (by \cite[Corollary 1.4]{O} and \cite[Proposition 2.1]{BD}), and thus $x^2$ and $x^3$ are units of $D$. Therefore, $x=x^{-2}x^3$ is a unit of $D$, and hence $x\in D$.

(7) $\Rightarrow$ (6): By \cite[Lemma 3.3.3]{GKR}, we have that $M$ is the maximal ideal of $\widehat{D}$. If $x\in\widehat{D}$, then $xM\subseteq M$ (since $M$ is an ideal of $\widehat{D}$). It is straightforward to show that $(M:M)\subseteq\widehat{D}$. We infer that $(M:M)=\widehat{D}$ is a $DVR$.

Now let the equivalent statements of Theorem~\ref{thm3} be satisfied. It remains to show that $D$ is a half-factorial $C$-domain and a Mori domain. It follows from \cite[Theorem 6.2]{AMO} that $D$ is a half-factorial domain. Obviously, $V$ is a subgroup of finite index of $V$ and $VM\subseteq\widehat{D}M=(M:M)M\subseteq M$. It follows from \cite[Corollary 2.8]{HHK} and \cite[Corollary 2.9.8]{GH} that $D$ is a $C$-domain. Moreover, $D$ is a Mori domain by \cite[Proposition 2.5.1]{HHK}.
\end{proof}

We want to point out that a local one-dimensional $OAF$-domain need not be Noetherian. Let $K\subseteq L$ be a field extension such that $[L:K]=\infty$ and let $D=K+XL[\![X]\!]$. Then $D$ is a local one-dimensional domain with maximal ideal $M=XL[\![X]\!]$ and $(M:M)=L[\![X]\!]$ is a $DVR$ with maximal ideal $M$. Consequently, $D$ is an $OAF$-domain by Theorem~\ref{thm3}. Since $[L:K]=\infty$, it follows that $D$ is not Noetherian.

An integral domain $D$ is called a {\it Cohen-Kaplansky domain} if $D$ is atomic and $D$ has only finitely many irreducible elements up to associates. It follows from \cite[Example 6.7]{AMO} that there exists a local half-factorial Cohen-Kaplansky domain with maximal ideal $M$ for which $M^2$ is not universal. We infer by Theorem~\ref{thm3} that the aforementioned domain is not an $OAF$-domain.

\begin{theorem}\label{thm4} Let $R$ be a ring with Jacobson radical $M$. The following statements are equivalent.
\begin{enumerate}
\item[\textnormal{(1)}] $R$ is an $OAF$-ring.
\item[\textnormal{(2)}] Each proper $2$-generated ideal of $R$ has an $OA$-factorization.
\item[\textnormal{(3)}] $\dim(R)\leq 1$ and each proper principal ideal has an $OA$-factorization.
\item[\textnormal{(4)}] $R$ satisfies one of the following conditions.
\begin{enumerate}
\item[\textnormal{(A)}] $R$ is a general $ZPI$-ring.
\item[\textnormal{(B)}] $R$ is a local domain, $M^2$ is divided and $\bigcap_{n\in\mathbb{N}} M^n=\boldsymbol{0}$.
\item[\textnormal{(C)}] $R$ is local, $M^2$ is divided and $M$ is nilpotent.
\end{enumerate}
\end{enumerate}
\end{theorem}

\begin{proof} (1) $\Rightarrow$ (2): This is obvious.

(2) $\Rightarrow$ (3): This is an immediate consequence of Theorem~\ref{thm2}.

(3) $\Rightarrow$ (4): First let each $OA$-ideal of $R$ be a prime ideal. Then $R$ is a $\pi$-ring. By \cite[Theorems 39.2, 46.7, and 46.11]{GI}, $R$ is a general $ZPI$-ring. Now let there be an $OA$-ideal of $R$ which is not a prime ideal. It follows from Lemma~\ref{lem1} that $R$ is local with maximal ideal $M$ and $M$ is not idempotent. Note that if $x\in M\setminus M^2$, then $xR$ cannot be a product of more than one $OA$-ideal, and hence $xR$ is an $OA$-ideal.

\textsc{Case} 1: $R$ is zero-dimensional. Let $x\in M\setminus M^2$. Then $xR$ is an $OA$-ideal. We infer by Lemma~\ref{lem1}(2) that $M^2\subseteq xR$. Consequently, $M^2$ is divided. It follows from Lemma~\ref{lem1} that $M^2\subseteq I$ for each $OA$-ideal $I$ of $R$. Since $\boldsymbol{0}$ is a product of $OA$-ideals, we have that $\boldsymbol{0}$ contains a power of $M$. This implies that $M$ is nilpotent.

\textsc{Case} 2: $R$ is one-dimensional. It follows from Proposition~\ref{prop4} that $R$ is an integral domain, and hence $\bigcap_{n\in\mathbb{N}} M^n=\boldsymbol{0}$ by Lemma~\ref{lem4}. It remains to show that $M^2$ is divided. Let $x\in R\setminus M^2$. Without restriction let $x$ be a nonunit. Then $xR$ is an $OA$-ideal. By Lemma~\ref{lem1}(2) we have that $M^2\subseteq xR$.

(4) $\Rightarrow$ (1): Clearly, every general $ZPI$-ring is an $OAF$-ring. The rest follows from Lemma~\ref{lem5}.
\end{proof}

\begin{corollary}\label{cor3} Let $R$ be a ring with Jacobson radical $M$. The following statements are equivalent.
\begin{enumerate}
\item[\textnormal{(1)}] Each proper principal ideal of $R$ has an $OA$-factorization.
\item[\textnormal{(2)}] $R$ is a $\pi$-ring or an $OAF$-ring.
\item[\textnormal{(3)}] $R$ satisfies one of the following conditions.
\begin{enumerate}
\item[\textnormal{(A)}] $R$ is a $\pi$-ring.
\item[\textnormal{(B)}] $R$ is a local domain, $M^2$ is divided and $\bigcap_{n\in\mathbb{N}} M^n=\boldsymbol{0}$.
\item[\textnormal{(C)}] $R$ is local, $M^2$ is divided and $M$ is nilpotent.
\end{enumerate}
\end{enumerate}
\end{corollary}

\begin{proof} (1) $\Rightarrow$ (2): If $R$ is not local, then $R$ is a $\pi$-ring by Remark~\ref{rem1}(2). Now let $R$ be local. If $\dim(R)\geq 2$, then $R$ is a unique factorization domain by Proposition~\ref{prop4}, and hence $R$ is a $\pi$-ring. If $\dim(R)\leq 1$, then $R$ is an $OAF$-ring by Theorem~\ref{thm4}.

(2) $\Rightarrow$ (1): This is obvious.

(2) $\Leftrightarrow$ (3): This is an immediate consequence of Theorem~\ref{thm4} and the fact that every general $ZPI$-ring is a $\pi$-ring.
\end{proof}

\begin{corollary}\label{cor4} Let $R$ be a ring with Jacobson radical $M$. The following statements are equivalent.
\begin{enumerate}
\item[\textnormal{(1)}] Each proper principal ideal of $R$ is a product of principal $OA$-ideals.
\item[\textnormal{(2)}] $R$ satisfies one of the following conditions.
\begin{enumerate}
\item[\textnormal{(A)}] $R$ is a unique factorization ring.
\item[\textnormal{(B)}] $R$ is a local domain, $M^2$ is divided and $\bigcap_{n\in\mathbb{N}} M^n=\boldsymbol{0}$.
\item[\textnormal{(C)}] $R$ is local, $M^2$ is divided and $M$ is nilpotent.
\end{enumerate}
\end{enumerate}
\end{corollary}

\begin{proof} (1) $\Rightarrow$ (2): If $R$ is not local, then $R$ is a unique factorization ring by Remark~\ref{rem1}(3). If $R$ is local, then the statement follows from Corollary~\ref{cor3} and the fact that every local $\pi$-ring is a unique factorization ring (\cite[Corollary 2.2]{AMA}).

(2) $\Rightarrow$ (1): Obviously, if $R$ is a unique factorization ring, then each proper principal ideal of $R$ is a product of principal $OA$-ideals. The rest is an immediate consequence of Lemma~\ref{lem5}.
\end{proof}

In Lemma~\ref{lem1}, we saw that if $R$ is a local ring with maximal ideal $M$ and $I$ is an ideal of $R$ such that $M^2\subseteq I$, then $I$ is an $OA$-ideal of $R$. Now we will give a characterization of the rings for which every proper (principal) ideal is an $OA$-ideal.

\begin{proposition}\label{prop6} Let $R$ be a ring with Jacobson radical $M$. The following statements are equivalent.
\begin{enumerate}
\item[\textnormal{(1)}] Every proper ideal of $R$ is an $OA$-ideal.
\item[\textnormal{(2)}] Every proper principal ideal of $R$ is an $OA$-ideal.
\item[\textnormal{(3)}] $R$ is local and $M^2=\boldsymbol{0}$.
\end{enumerate}
\end{proposition}

\begin{proof} (1) $\Rightarrow$ (2): This is obvious.

(2) $\Rightarrow$ (3): Assume that $R$ is not local. Then every proper principal ideal of $R$ is a prime ideal by Lemma~\ref{lem1}(1). Consequently, $R$ is an integral domain. If $x\in R$ is a nonunit, then $x^2R$ is a prime ideal, and hence $x^2R=xR$ and $x=0$. Therefore, $R$ is a field, a contradiction. This implies that $R$ is local with maximal ideal $M$. We infer by Lemma~\ref{lem1}(2) that $\boldsymbol{0}$ is a prime ideal or $M^2=\boldsymbol{0}$.

Assume that $M^2\not=\boldsymbol{0}$. Then $R$ is an integral domain and there is some nonzero $x\in M^2$. It follow from Lemma~\ref{lem1}(2) that $x^2R$ is a prime ideal or $M^2\subseteq x^2R$. If $x^2R$ is a prime ideal, then $x^2R=xR$. If $M^2\subseteq x^2R$, then $M^2\subseteq x^2R\subseteq xR\subseteq M^2$, and thus $x^2R=xR$. In any case we have that $x^2R=xR$, and hence $x$ is a unit (since $x$ is regular), a contradiction.

(3) $\Rightarrow$ (1): This is an immediate consequence of Lemma~\ref{lem1}(2).
\end{proof}

\section{$OA$-factorization properties and trivial ring extensions}

Let $A$ be a ring and $E$ be an $A$-module. Then $A\propto E$, the {\it trivial} ({\it ring}) {\it extension of} $A$ {\it by} $E$, is the ring whose additive structure is that of the external direct sum $A\oplus E$ and whose multiplication is defined by $(a,e)(b,f)=(ab,af+be)$ for all $a,b\in A$ and all $e,f\in E$. (This construction is also known by other terminology and other notation, such as the {\it idealization} $A(+)E$.) The basic properties of trivial ring extensions are summarized in the textbooks \cite{G,H}. Trivial ring extensions have been studied or generalized extensively, often because of their usefulness in constructing new classes of examples of rings satisfying various properties (cf. \cite{AW,BKM,KM}). We say that $E$ is {\it divisible} if $E=aE$ for each regular element $a\in A$.

We start with the following lemma.

\begin{lemma}\label{lem6} Let $A$ be a ring, $I$ be an ideal of $A$ and $E$ be an $A$-module. Let $R=A\propto E$ be the trivial ring extension of $A$ by $E$.
\begin{enumerate}
\item[\textnormal{(1)}] $I\propto E$ is an $OA$-ideal of $R$ if and only if $I$ is an $OA$-ideal of $A$.
\item[\textnormal{(2)}] Assume that $A$ contains a nonunit regular element and $E$ is a divisible $A$-module. Then the $OA$-ideals of $R$ have the form $L\propto E$ where $L$ is an $OA$-ideal of $A$.
\end{enumerate}
\end{lemma}

\begin{proof} (1) This follows immediately from \cite[Theorem 2.20]{YNN}.

(2) Let $J$ be an $OA$-ideal of $R$. Our aim is to show that $\boldsymbol{0}\propto E\subseteq J$. Let $e\in E$ and let $a\in A$ be a nonunit regular element. Then $e=af$ for some $f\in E$ and thus $(a,0)(0,f)(0,e)=(0,0)\in J$. Since $J$ is an $OA$-ideal, we conclude that $(a,0)(0,f)=(0,e)\in J$ or $(0,e)\in J$ which implies that $\boldsymbol{0}\propto E\subseteq J$. Therefore, $J=L\propto E$ with $L=\{b\in A\mid (b,g)\in J$ for some $g\in E\}$ and $L$ is an ideal of $A$ by \cite[Theorems 3.1 and 3.3(1)]{AW}. Now the result follows from (1).
\end{proof}

\begin{corollary}\label{cor5} Let $A$ be an integral domain that is not a field, $E$ be a divisible $A$-module and $R=A\propto E$. Then the $OA$-ideals of $R$ have the form $I\propto E$ where $I$ is an $OA$-ideal of $A$.
\end{corollary}

Next, we study the transfer of the $OAF$-ring property to the trivial ring extension.

\begin{theorem}\label{thm5} Let $A$ be a ring with Jacobson radical $M$, $E$ be an $A$-module and $R=A\propto E$.
\begin{enumerate}
\item[\textnormal{(1)}] $R$ is an $OAF$-ring if and only if one of the following conditions is satisfied.
\begin{enumerate}
\item[\textnormal{(A)}] $A$ is a general $ZPI$-ring, $E$ is cyclic and the annihilator of $E$ is a $($possibly empty$)$ product of idempotent maximal ideals of $A$.
\item[\textnormal{(B)}] $A$ is local, $M^2$ is divided, $E=\boldsymbol{0}$ and either $M$ is nilpotent or $A$ is a domain with $\bigcap_{n\in\mathbb{N}} M^n=\boldsymbol{0}$.
\item[\textnormal{(C)}] $A$ is local, $M^2=\boldsymbol{0}$, $ME=aE$ for each nonzero $a\in M$ and $ME=Mx$ for each $x\in E\setminus ME$.
\end{enumerate}
In particular, if $R$ is an $OAF$-ring, then $A$ is an $OAF$-ring.
\item[\textnormal{(2)}] Every proper ideal of $R$ is an $OA$-ideal if and only if $A$ is local, $M^2=\boldsymbol{0}$ and $ME=\boldsymbol{0}$.
\end{enumerate}
\end{theorem}

\begin{proof} (1) ($\Rightarrow$) First let $R$ be an $OAF$-ring. By Theorem~\ref{thm4}, it follows that (a) $R$ is a general $ZPI$-ring or (b) $R$ is local with maximal ideal $N$, $N^2$ is divided and ($N$ is nilpotent or $R$ is a domain such that $\bigcap_{n\in\mathbb{N}} N^n=\boldsymbol{0}$). If $R$ is a general $ZPI$-ring, then condition (A) is satisfied by \cite[Theorem 4.10]{AW}.

From now on let $R$ be local with maximal ideal $N$ such that $N^2$ is divided. Observe that $A$ is local with maximal ideal $M$ and $N=M\propto E$ by \cite[Theorem 3.2(1)]{AW}. If $R$ is a domain such that $\bigcap_{n\in\mathbb{N}} N^n=\boldsymbol{0}$, then $E=\boldsymbol{0}$ (for if $z\in E$ is nonzero, then $(0,z)$ is a nonzero zero-divisor of $R$), and hence $A\cong R$ is a domain, $M^2$ is divided and $\bigcap_{n\in\mathbb{N}} M^n=\boldsymbol{0}$.

Now let $N$ be nilpotent. If $E=\boldsymbol{0}$, then $A\cong R$, and thus $M^2$ is divided and $M$ is nilpotent. From now on let $E$ be nonzero. There is some $k\in\mathbb{N}$ such that $N^k=\boldsymbol{0}$. Note that $N^2=M^2\propto ME$ and $N^k=M^k\propto M^{k-1}E$, and thus $M^k=\boldsymbol{0}$. Since $N^2$ is divided, we have that $\boldsymbol{0}\propto E\subseteq N^2$ or $N^2\subseteq\boldsymbol{0}\propto E$. If $\boldsymbol{0}\propto E\subseteq N^2$, then $E=ME$, and hence $E=M^kE=\boldsymbol{0}$, a contradiction. Therefore, $N^2\subseteq\boldsymbol{0}\propto E$, which implies that $M^2=\boldsymbol{0}$.

Let $a\in M$ be nonzero. Then $(a,0)\not\in N^2$, and hence $N^2\subseteq (a,0)R=aA\propto aE$. Consequently, $ME\subseteq aE$, and thus $ME=aE$. Finally, let $x\in E\setminus ME$. Then $(0,x)\not\in N^2$. We infer that $N^2\subseteq (0,x)R=\boldsymbol{0}\propto Ax$. This implies that $ME\subseteq Ax$. If $ME\nsubseteq Mx$, then $bx\in ME$ for some unit $b\in A$, and hence $x\in ME$, a contradiction. It follows that $ME\subseteq Mx$, which clearly implies that $ME=Mx$.

($\Leftarrow$) Next we prove the converse. If condition (A) is satisfied, then $R$ is a general $ZPI$-ring by \cite[Theorem 4.10]{AW}, and thus $R$ is an $OAF$-ring. If condition (B) is satisfied, then $A$ is an $OAF$-ring by Theorem~\ref{thm4}, and hence $R\cong A$ is an $OAF$-ring. Now let condition (C) be satisfied. Set $N=M\propto E$. Then $R$ is local with maximal ideal $N$ by \cite[Theorem 3.2(1)]{AW}. By Theorem~\ref{thm4}, it suffices to show that $N$ is nilpotent and $N^2$ is divided. Since $M^2=\boldsymbol{0}$, we obtain that $N^3=M^3\propto M^2E=\boldsymbol{0}$, and thus $N$ is nilpotent. It remains to show that $N^2\subseteq (a,x)R$ for each $(a,x)\in R\setminus N^2$. Let $a\in A$ and $x\in E$ be such that $(a,x)\not\in N^2$. Since $N^2=\boldsymbol{0}\propto ME$, we have to show that $\boldsymbol{0}\propto ME\subseteq (a,x)R$. If $a$ is a unit of $A$, then $(a,x)$ is a unit of $R$ by \cite[Theorem 3.7]{AW} and the statement is clearly true. Let $z\in\boldsymbol{0}\propto ME$. Then $z=(0,y)$ for some $y\in ME$.

\textsc{Case} 1: $a$ is a nonzero nonunit. Since $ME=aE$, there is some $v\in E$ such that $y=av$. Observe that $z=(0,av)=(a,x)(0,v)\in (a,x)R$.

\textsc{Case} 2: $a=0$. Then $x\in E\setminus ME$ (since $(a,x)\not\in N^2$). Since $ME=Mx$, there is some $b\in M$ such that $y=bx$. It follows that $z=(0,bx)=(a,x)(b,0)\in (a,x)R$.

The in particular statement now follows from Theorem~\ref{thm4}.

(2) First let every proper ideal of $R$ be an $OA$-ideal. By Proposition~\ref{prop6}, we have that $R$ is local with maximal ideal $N$ and $N^2=\boldsymbol{0}$. It follows that $A$ is local with maximal ideal $M$ and $N=M\propto E$ by \cite[Theorem 3.2(1)]{AW}. Moreover, $\boldsymbol{0}=N^2=M^2\propto ME$, and hence $M^2=\boldsymbol{0}$ and $ME=\boldsymbol{0}$.

Conversely, let $A$ be local, $M^2=\boldsymbol{0}$ and $ME=\boldsymbol{0}$. Set $N=M\propto E$. Then $R$ is local with maximal ideal $N$ by \cite[Theorem 3.2(1)]{AW} and $N^2=M^2\propto ME=\boldsymbol{0}$. We infer by Proposition~\ref{prop6} that each proper ideal of $R$ is an $OA$-ideal.
\end{proof}

\begin{corollary}\label{cor6} Let $A$ be an integral domain, $E$ be a nonzero $A$-module and $R=A\propto E$. The following statements are equivalent.
\begin{enumerate}
\item[\textnormal{(1)}] $R$ is an $OAF$-ring.
\item[\textnormal{(2)}] $A$ is a field.
\item[\textnormal{(3)}] Every proper ideal of $R$ is an $OA$-ideal.
\end{enumerate}
\end{corollary}

\begin{proof} (1) $\Rightarrow$ (2): It follows from Theorem~\ref{thm5}(1) that $A$ is a general $ZPI$-ring and the annihilator of $E$ is a product of idempotent maximal ideals of $A$ or that $A$ is local with maximal ideal $M$ such that $M^2=\boldsymbol{0}$.

First let $A$ be a general $ZPI$-ring such that the annihilator of $E$ is a product of idempotent maximal ideals of $A$. Note that $A$ is a Dedekind domain, and thus the only proper idempotent ideal of $A$ is the zero ideal. Since $E$ is nonzero, the annihilator of $E$ is a proper ideal of $A$, and hence $A$ must possess an idempotent maximal ideal. We infer that the zero ideal is a maximal ideal of $A$, and thus $A$ is a field.

Now let $A$ be local with maximal ideal $M$ such that $M^2=\boldsymbol{0}$. Since $A$ is an integral domain, it follows that $M=\boldsymbol{0}$, and hence $A$ is a field.

(2) $\Rightarrow$ (3): Set $M=\boldsymbol{0}$. Then $A$ is local with maximal ideal $M$, $M^2=\boldsymbol{0}$ and $ME=\boldsymbol{0}$. Now the statement follows from Theorem~\ref{thm5}(2).

(3) $\Rightarrow$ (1): This is obvious.
\end{proof}

\begin{remark}\label{rem2} In general, if $A$ is an $OAF$-ring and $E$ is an $A$-module, then $A\propto E$ need not be an $OAF$-ring. Indeed, let $A$ be an $OAF$-domain that is not a field and let $E$ be a nonzero $A$-module. By Corollary~\ref{cor6}, $A\propto E$ is not an $OAF$-ring.
\end{remark}

\begin{corollary}\label{cor7} Let $A$ be a local ring with maximal ideal $M$ and $E$ be a nonzero $A$-module such that $ME=\boldsymbol{0}$. Set $R=A\propto E$. The following statements are equivalent.
\begin{enumerate}
\item[\textnormal{(1)}] $R$ is an $OAF$-ring.
\item[\textnormal{(2)}] $M^2=\boldsymbol{0}$.
\item[\textnormal{(3)}] Every proper ideal of $R$ is an $OA$-ideal.
\end{enumerate}
\end{corollary}

\begin{proof} (1) $\Rightarrow$ (2): Assume that $M^2\not=\boldsymbol{0}$. By Theorem~\ref{thm5}(1), $A$ is a local general $ZPI$-ring and $M$ is idempotent (since the annihilator of $E$ is a nonempty product of idempotent maximal ideals of $A$ and $M$ is the only maximal ideal of $A$). We infer by \cite[Corollary 9.11]{LM} that $A$ is a Dedekind domain or each proper ideal of $A$ is a power of $M$ (because local rings are indecomposable). If $A$ is a Dedekind domain, then clearly $M^2=M=\boldsymbol{0}$ (since $M$ is idempotent and a Dedekind domain has no nonzero proper idempotent ideals). Moreover, if every proper ideal of $A$ is a power of $M$, then again $M^2=M=\boldsymbol{0}$ (since $M$ is idempotent). In any case, we obtain that $M^2=\boldsymbol{0}$, a contradiction.

(1) $\Leftarrow$ (2) $\Leftrightarrow$ (3): This follows from Theorem~\ref{thm5}.
\end{proof}

\begin{example}\label{ex1} Let $A$ be a local principal ideal ring with maximal ideal $M$ such that $A$ is not a field and $M^2=\boldsymbol{0}$ (e.g. $A=\mathbb{Z}/4\mathbb{Z}$). Set $R=A\propto A$. Then $R$ is an $OAF$-ring, and yet not every proper ideal of $R$ is an $OA$-ideal.
\end{example}

\begin{proof} Since $M\not=\boldsymbol{0}$, it follows from Theorem~\ref{thm5}(2) that not every proper ideal of $R$ is an $OA$-ideal. By Theorem~\ref{thm5}(1) it remains to show that $M=aA$ for each nonzero $a\in M$ and $M=Mx$ for each $x\in A\setminus M$. Note that $M=zA$ for some $z\in M$. If $a\in M$ is nonzero, then $a=zb$ for some $b\in A$. Clearly, $b\not\in M$, and thus $b$ is a unit of $A$, which clearly implies that $M=zA=aA$. Finally, if $x\in A\setminus M$, then $x$ is a unit of $A$, and thus $M=Mx$.
\end{proof}

\begin{remark}\label{rem3} Let $A$ be a ring with Jacobson radical $M$, $E$ be an $A$-module and $R=A\propto E$. Then each proper principal ideal of $R$ has an $OA$-factorization if and only if one of the following conditions is satisfied.
\begin{enumerate}
\item[\textnormal{(1)}] $A$ is a $\pi$-ring, $E$ is cyclic and the annihilator of $E$ is a $($possibly empty$)$ product of idempotent maximal ideals of $A$.
\item[\textnormal{(2)}] $A$ is local, $M^2$ is divided, $E=\boldsymbol{0}$ and either $M$ is nilpotent or $A$ is a domain with $\bigcap_{n\in\mathbb{N}} M^n=\boldsymbol{0}$.
\item[\textnormal{(3)}] $A$ is local, $M^2=\boldsymbol{0}$, $ME=aE$ for each nonzero $a\in M$ and $ME=Mx$ for each $x\in E\setminus ME$.
\end{enumerate}
\end{remark}

\begin{proof} This can be proved along similar lines as Theorem~\ref{thm5}(1) by using Corollary~\ref{cor3} and \cite[Theorems 3.2(1) and 4.10]{AW}.
\end{proof}

\bigskip
\noindent {\bf ACKNOWLEDGEMENTS.} We want to thank the referee for many helpful suggestions and comments which improved the quality of this paper. The fourth-named author was supported by the Austrian Science Fund FWF, Project Number J4023-N35. The first three authors dedicate this work to their Professor Faycal Lamrini for his retirement.

\end{document}